\documentclass[11 pt]{amsart}
\usepackage{color}
\usepackage{verbatim}
\usepackage{amssymb}
\usepackage{amscd}
\usepackage{amsmath}
\usepackage{amsthm}
\usepackage{eucal}
\usepackage{setspace}
\usepackage{url}
\usepackage[utf8]{inputenc}
\usepackage{booktabs}

\allowdisplaybreaks
\newcommand{\Tr}{\mathrm{Tr}}

\renewcommand{\Im}{\mathrm{Im}}

\newcommand{\Sp}{\mathrm{Sp}}
\renewcommand{\H}{\mathbb H}
\newcommand{\T}[1]{{}^t{{#1}}}
\newcommand{\field}{K}
\newcommand{\A}{{\mathbb A}}
\renewcommand{\a}{{\mathfrak a}}
\newcommand{\Q}{{\mathbb Q}}
\newcommand{\Z}{{\mathbb Z}}
\newcommand{\R}{{\mathbb R}}
\newcommand{\C}{{\mathbb C}}
\newcommand{\bs}{\backslash}

\newcommand{\GL}{{\rm GL}}
\newcommand{\PGL}{{\rm PGL}}
\newcommand{\SL}{{\rm SL}}

\newcommand{\bc}{{\rm BC}}

\newcommand{\GSp}{{\rm GSp}}
\newcommand{\PGSp}{{\rm PGSp}}

\newcommand{\St}{{\rm St}}

\newcommand{\sq}{\mathfrak{S}}

\newcommand{\disc}{{\rm disc}}

\DeclareMathOperator{\Cl}{Cl}
\newcommand{\mat}[4]{{\setlength{\arraycolsep}{0.5mm}\left(\begin{array}{cc}#1&#2\\#3&#4\end{array}\right)}}

\newcommand{\forget}[1]{}

\newtheorem{lemma}{Lemma}[section]
\newtheorem{theorem}{Theorem}
\newtheorem*{thm}{Theorem}

\newtheorem{proposition}[lemma]{Proposition}

\theoremstyle{remark}
\newtheorem{remark}[lemma]{Remark}

\begin{document}

\bibliographystyle{plain}

\title[Yoshida lifts and simultaneous non-vanishing]{Yoshida lifts and simultaneous non-vanishing of dihedral twists of modular L-functions}

\author{Abhishek Saha}
\address{Departments of Mathematics \\
  University of Bristol\\
  Bristol BS81SN \\
  UK} \email{abhishek.saha@bris.ac.uk}

\author{Ralf Schmidt}
\address{Department of Mathematics
\\ University of Oklahoma\\ Norman\\
   OK 73019, USA}
\email{rschmidt@math.ou.edu}

\begin{abstract}
 Given elliptic modular forms $f$ and $g$ satisfying certain conditions on their weights and levels, we prove (a quantitative version of the statement) that there exist infinitely many imaginary quadratic fields $K$ and characters $\chi$ of the ideal class group ${\rm Cl}_K$ such that $L(\frac12,\bc_{K}(f)\times\chi)\neq0$ and $L(\frac12,\bc_{K}(g)\times\chi)\neq0$. The proof is based on a non-vanishing result for Fourier coefficients of Siegel modular forms combined with the theory of Yoshida liftings.
\end{abstract}

 \maketitle

\let\thefootnote\relax

\section{Introduction}
Let $K$ be an  imaginary quadratic field of discriminant $-d$. We denote its ideal class group by ${\Cl_K}$ and the  group of ideal class characters by $\widehat{\Cl_K}$. For any $\chi$ in $\widehat{\Cl_K}$ and $f$ a holomorphic newform with trivial nebentypus as in~\cite{at-le}, one can form the $L$-function $L(s, \pi_f \times \theta_\chi)$; this is the Rankin-Selberg convolution of the automorphic representation $\pi_f$ attached to $f$ and the $\theta$-series $$\theta_\chi(z) = \sum_{0 \ne \a \subset O_K}\chi(\a) e(N(\a)z).$$
Here, $\theta_\chi$ is a holomorphic modular form of weight $1$ and nebentypus $\left(\frac{-d}{*} \right)$ on $\Gamma_0(d)$; it is a cusp form if and only if $\chi^2 \ne 1$. We remark here that $L(s, \pi_f \times \theta_\chi) = L(s, \bc_{K}(\pi_f) \times \chi) $ where $\bc_{K}$ denotes base-change to $K$.

The problem of studying the non-vanishing of the central values $L(\frac12, \pi_f \times \theta_\chi)$ arises naturally in several contexts, and a considerable amount of work has been done in this direction. We note in particular the paper of Michel and Venkatesh~\cite{micven07} which proves that given a cusp form $f$ (satisfying certain conditions on weight and level) and an imaginary quadratic field $K$ of discriminant $-d$, there exist asymptotically at least $d^{\frac{1}{2700} - \epsilon}$ characters $\chi \in \widehat{\Cl_K}$ such that $L(\frac12, \pi_f \times \theta_\chi) \ne 0$. The introduction to~\cite{micven07} has a review of several other papers on related questions and a summary of the methods available.

In this paper, we prove a \emph{simultaneous} non-vanishing result for $L(\frac12, \pi_f \times \theta_\chi)$, $L(\frac12, \pi_g \times \theta_\chi) $ for two fixed forms $f$, $g$ (but varying $K$ and $\chi$) under certain hypotheses.

\begin{theorem}\label{th:simulnonvan}Let $k>1$ be an odd integer. Let $N_1$, $N_2$ be two positive, squarefree integers such that $M = \gcd(N_1, N_2)>1$. Let $f$ be a holomorphic newform of weight $2k$ on $\Gamma_0(N_1)$ and $g$ be a holomorphic newform of weight $2$ on $\Gamma_0(N_2)$. Assume that for all primes $p$ dividing $M$ the Atkin-Lehner eigenvalues of $f$ and $g$ coincide. Then there exists an imaginary quadratic field $K$ and a character $\chi \in  \widehat{\Cl_K}$ such that $L(\frac12, \pi_f \times \theta_\chi) \ne 0$ and $L(\frac12, \pi_g \times \theta_\chi) \ne 0$. In fact, if $D(f,g)$ is the set of $d$ satisfying the following conditions:

 \begin{enumerate}
 \item $d>0$ is an odd, squarefree integer and $-d$ is a fundamental discriminant,
  \item \label{part 2} There exists an ideal class group character $\chi$ of $K=\Q(\sqrt{-d})$ such that  $L(\frac12, \pi_f \times \theta_\chi) \neq 0$ and $L(\frac12, \pi_g \times \theta_\chi) \neq 0$,
 \end{enumerate}

then, for any $0<\delta <\frac58$, one has the lower bound \footnote{Recall that $A(X) \gg_{a,b,..} B(X)$ means that there exist constants $C>0$, $D>0$, which depend only on $a,b,..$, such that $A(X) > C |B(X)|$ for all $X>D$.}
\begin{equation} \label{bounddelta}|\{0<d<X, \ d\in  D(f,g) \}| \gg_{f,g,\delta} X^\delta.\end{equation}

\end{theorem}

 Note that if we could show that for some $K$, the trivial character is a suitable choice for $\chi$, we would solve the long-standing problem of showing the existence of a quadratic Dirichlet character whose twist with the central $L$-values of $f$, $g$ are simultaneously non-zero. However, because of the nature of our method, we cannot get any handle on the $\chi$ that are good for our purposes, nor can we give any quantitative bounds on how many such $\chi$ exist for each $K$.

 Our method involves Siegel modular forms, Jacobi forms, and classical holomorphic forms of half-integral weight. First, we lift the pair $(f,g)$ via the theta correspondence to a Siegel cusp form of degree 2 and level $N$. Such lifts are traditionally known as Yoshida lifts (after Hiroyuki Yoshida who first investigated such forms in~\cite{yosh1980}) and have been studied extensively by B{\"o}cherer and Schulze-Pillot~\cite{bocsch, bocsch1991, bocsch1994, bocsch1997}. In fact, the Yoshida lift is a certain case of Langlands functoriality; see Sect.~\ref{yoshidarepsec} for more details.

Via the ``pullback" of Bessel periods from \cite[Theorem 3]{prasadbighash} and the formula of Waldspurger, Theorem~\ref{th:simulnonvan} reduces to showing that the Yoshida lift attached to $(f,g)$ has many non-vanishing Fourier coefficients of fundamental discriminant. This turns out to be a special case of the other main result of this paper, Theorem~\ref{th:nonvanfouriersiegel}, which asserts that any Siegel cusp form of degree 2 and squarefree level that is an eigenfunction of certain Hecke operators has many non-zero fundamental Fourier coefficients. The proof of Theorem~\ref{th:nonvanfouriersiegel} --- which exploits the Fourier-Jacobi expansion of Siegel forms and the relation between Jacobi forms and holomorphic modular forms of half-integral weight --- involves only minor modifications to the proof of the main result of~\cite{sahafund}, where a version of this theorem in the case of full level was proved.

After some basic facts and definitions, Theorem~\ref{th:nonvanfouriersiegel} is stated in Sect.\ \ref{nonvanishingsec} below. We explain how its proof follows from a statement about half-integral weight modular forms, and continue to prove this half-integral weight result in Sect.\ \ref{halfintegralsec}. Then, in Sect.\ \ref{yoshidasec}, we turn to Yoshida liftings, starting with some general facts on the relationship between Siegel modular forms of degree $2$ and automorphic representations of the group $\GSp_4$. This is followed by a brief survey of the representation-theoretic construction of the Yoshida lifting due to Roberts \cite{rob2001}. Combined with some local results about representations of the non-archimedean $\GSp_4$, we explain how this leads to an alternative proof of the existence of the classical Yoshida liftings constructed in \cite{bocsch1991, bocsch1997}. The alternative proof comes with a few additional benefits, which will be used in Sect.\ \ref{besselsec}. We start this final section by giving some background on Bessel models and their relationship with Fourier coefficients. Finally, combining Theorem~\ref{th:nonvanfouriersiegel}, Yoshida liftings and \cite[Thm.\ 3]{prasadbighash}, we prove Theorem \ref{th:simulnonvan}. Note that while our method gives a lower bound on the number of non-vanishing twists, it does not give a lower bound on the size of the non-vanishing $L$-value itself.

We say a few words about the restrictions on $f$ and $g$ in Theorem~\ref{th:simulnonvan}. The conditions that $N_1, N_2$ are squarefree and that the Atkin-Lehner eigenvalues of $f$ and $g$ coincide are needed to ensure that there exists a holomorphic Yoshida lift attached to $(f,g)$ with respect to a Siegel-type congruence subgroup $\Gamma_0^{(2)}(N) \subset \Sp_4(\Z)$ of squarefree level $N$. Indeed, our key result (Theorem~\ref{th:nonvanfouriersiegel}) on non-vanishing fundamental Fourier coefficients is only proved for Siegel cusp forms with respect to such congruence subgroups. However, even if these two conditions are removed, $(f,g)$ will still have a Yoshida lift (possibly with respect to some other congruence subgroup) provided that there is a prime $p$ dividing $\gcd(N_1, N_2)$ such that $\pi_f$, $\pi_g$ are both discrete series at $p$.\footnote{The restriction that there is a prime dividing $\gcd(N_1, N_2)$ where $\pi_f$, $\pi_g$ are both discrete series will probably be very difficult to remove by our method, because without this condition there are no Jacquet-Langlands transfers and hence no (holomorphic) Yoshida lifts. It is conceivable that one could still consider the ``Fourier coefficients" of the non-holomorphic Yoshida lift in this setup and prove a non-vanishing result for those.} So, an analogue of Theorem~\ref{th:nonvanfouriersiegel} for Siegel cusp forms with respect to more general congruence subgroups will allow us to remove some of the restrictions on $f$ and $g$. This is currently work in progress by J. Marzec at the University of Bristol. To remove the restriction on the weight of $g$ would require us to extend Theorem~\ref{th:nonvanfouriersiegel} to vector valued Siegel cusp forms, which seems possible in principle.

A word about the exponent $\frac58$ in Theorem~\ref{th:simulnonvan}. Let $\theta$ be a real number such that given any $\epsilon>0$ and a cusp form $f$ of weight $k + \frac12$ with $k\ge1$ we have $\widetilde{a}(f,n) \ll_{f, \epsilon} n^{\theta + \epsilon}$; here $n$ varies over squarefree integers coprime to the level and $\widetilde{a}(f,n)$ denotes the normalized Fourier coefficient. Then what we really prove is that~\eqref{bounddelta} is valid for any $0 < \delta < 1-2\theta$. The first non-trivial bound for $\theta$ ($= \frac{3}{14}$) was obtained by Iwaniec~\cite{iwanfourier}.  In Theorem~\ref{th:simulnonvan} we have used $\theta = \frac{3}{16}$ which is due to Bykovski{\u\i}~\cite{byko}; see also the papers of Blomer--Harcos~\cite{blomer-harcos08} and Conrey--Iwaniec~\cite{conrey-iwaniec}.

As for related work, we have already mentioned the paper of Michel and Venkatesh~\cite{micven07}. The more general problem of non-vanishing of twists of automorphic $L$-functions has a long history. The book of Murty--Murty~\cite{murty-murty}, which brings together some of the main techniques and results in the area, is a good reference; see also the introduction to Ono--Skinner~\cite{ono-skinner}. There are only a few simultaneous non-vanishing results available in the literature. An interesting example is the result of Michel and Vanderkam~\cite{micvan} where families of \emph{three} different $\GL_1 \times \GL_2$ $L$-functions are considered. Closely related to the present work is a paper of Prasad and Takloo-Bighash on Bessel models where a similar non-vanishing result is proved~\cite[Corollary 13.3]{prasadbighash}; however, in their result, the twisting character $\chi$ can be any Hecke character of $K$ (of possible high conductor) and it does not seem possible to give an \emph{effective} bound on this conductor in terms of $f$, $g$ by their method. For many arithmetic applications, it is necessary to know the existence of non-vanishing twists by characters whose conductor can be effective bounded, and the ideal scenario is when an unramified twist exists, as in Theorem~\ref{th:simulnonvan}.

\section*{Acknowledgements} This paper grew out of a suggestion made by Ramin Takloo-Bighash to the first author. We thank him for that, as well as for his encouragement and feedback. We also thank Dinakar Ramakrishnan for some helpful suggestions.

\section{Siegel cusp forms of degree 2}\label{siegelsection}

\subsection{Preliminaries}For any commutative ring $R$ and positive integer $n$, let $M_n(R)$
  denote the ring of $n$ by $n$ matrices with entries in $R$ and $\GL_n(R)$ denote the group of invertible matrices.  If
  $A\in M_n(R)$, we let $\T{A}$ denote its transpose. We let $M_n^{\rm sym}(R)$ denote the additive group of symmetric matrices in  $M_n(R)$. We say that a matrix in $M_n^{\rm sym}(\Z)$ is semi-integral if
  it has integral diagonal entries and half-integral off-diagonal
  ones. We let $\Lambda_n \subset M_n^{\rm sym}(\Z)$ denote the set of symmetric, semi-integral, positive-definite matrices of
size $n$.

Denote by $J$ the $4$ by $4$ matrix given by
$$
J =
\begin{pmatrix}
0 & I_2\\
-I_2 & 0\\
\end{pmatrix}.
$$ where $I_2$ is the identity matrix of size 2. Define the algebraic groups
   $\GSp_4$ and $\Sp_4$ over $\Z$ by
$$\GSp_4(R) = \{g \in \GL_4(R) \; | \; \T{g}Jg =
  \mu_2(g)J,\:\mu_2(g)\in R^{\times}\},$$
$$
\Sp_4(R) = \{g \in \GSp_4(R) \; | \; \mu_2(g)=1\},
$$
for any commutative ring $R$. The group $\GSp_4$ will be denoted by the letter $G$.

The Siegel upper-half space of degree 2 is defined by
$$
\H_2 = \{ Z \in M_2(\C)\;|\;Z =\T{Z},\ \Im(Z)
  \text{ is positive definite}\}.
$$
We define
$$
 g \langle Z\rangle = (AZ+B)(CZ+D)^{-1}\qquad\text{for }
 g=\begin{pmatrix} A&B\\ C&D \end{pmatrix} \in \Sp_4(\R),\;Z\in \H_2.
$$
We let $J(g,Z) = CZ + D$ and use $i_2$ to denote the point $\begin{pmatrix}i&\\& i \end{pmatrix} \in \H_2$.

For any positive integer $N$,
define
\begin{equation}\label{Gamma0defeq}
\Gamma_0^{(2)}(N) := \left\{\begin{pmatrix}A&B\\ C&D \end{pmatrix} \in \Sp_4(\Z)\;|\;C \equiv 0 \pmod{N}\right\}.
\end{equation}

Let $S_k^{(2)}(N)$ denote the space of holomorphic functions $F$ on
$\H_2$ which satisfy the relation
\begin{equation}\label{siegeldefiningrel}
F(\gamma \langle Z\rangle) = \det(J(\gamma,Z))^k F(Z)
\end{equation}
for $\gamma \in \Gamma_0^{(2)}(N)$, $Z \in \H_2$, and vanish at all the
cusps. Elements of $S_k^{(2)}(N)$ are often referred to as Siegel cusp forms of degree (genus) 2, weight $k$ and level $N$.

\subsection{The Fourier and Fourier-Jacobi expansions}
It is well-known that any $F$ in $S_k^{(2)}(N)$ has a Fourier expansion \begin{equation}\label{siegelfourierexpansion}F(Z)
=\sum_{T \in \Lambda_2} a(F, T) e(\Tr(TZ)).
\end{equation}
Applying~\eqref{siegeldefiningrel} for $\gamma = \begin{pmatrix}A&\\&\T{A}^{-1}\end{pmatrix}$, where $A \in \GL_2(\Z)$, yields the relation \begin{equation}\label{fourierinvariance}a(F, T) = \det(A)^k\,a(F, \T{A}TA) \end{equation} for  $A \in \GL_2(\Z)$. In particular, if $k$ is even, then the Fourier coefficient $a(F, T)$ depends only on the $\GL_2(\Z)$ equivalence class of $T$.

The Fourier expansion~\eqref{siegelfourierexpansion} also immediately shows that any $F \in S_k^{(2)}(N)$ has a ``Fourier--Jacobi expansion"

\begin{equation}\label{fjexpand}F(Z) = \sum_{m > 0} \phi_m(\tau, z) e(m \tau')\end{equation} where we write $Z= \begin{pmatrix}\tau&z\\z&\tau' \end{pmatrix}$ and for each $m>0$,
\begin{equation}\label{jacobifourier}\phi_m(\tau, z) = \sum_{\substack{n,r \in \Z \\ 4nm> r^2}}a \left(F, \mat{n}{r/2}{r/2}{m}\right) e(n \tau) e(r z) \in J_{k,m}^{\text{cusp}}(N). \end{equation} Here $J_{k,m}^{\text{cusp}}(N)$ denotes the space of Jacobi cusp forms of weight $k$, level $N$ and index $m$; for details see~\cite{manickram}. If we put $c(n,r) = a \left(F, \mat{n}{r/2}{r/2}{m}\right)$, then~\eqref{jacobifourier} becomes

$$\phi_m(\tau, z) =  \sum_{\substack{n,r \in \Z \\ 4nm> r^2}} c(n,r) e(n \tau) e(r z),$$ and this is called the Fourier expansion of the Jacobi form $\phi_m$.

\subsection{The $U(p)$ operator} For each prime $p$ dividing $N$, there exists a Hecke operator $U(p)$ acting on the space $S_k^{(2)}(N)$. It can be most simply described by its action on Fourier coefficients,

\begin{equation}\label{upaction}F(Z) = \sum_{T \in \Lambda_2} a(F, T) e(\Tr(TZ)) \mapsto (U(p) F)(Z) =
\sum_{T \in \Lambda_2} a(F, pT) e(\Tr(TZ)).\end{equation}

When $N$ is squarefree, the operator $U(p)$ has been interpreted representation-theoretically in~\cite{sch} (where this operator is called $T_2(p)$). Furthermore, it has been proved by B{\"o}cherer~\cite{boch-up} that $U(p)$ is an invertible operator on the space $S_k^{(2)}(N)$ (we will however not need this fact).

\begin{lemma}\label{nonvanlemma} Let $F \in S_k^{(2)}(N)$ be an eigenfunction for the Hecke operators $U(p)$ for all $p$ dividing $N$. Suppose for some $N_1$ dividing $N$ and some $T \in \Lambda_2$ we have that $a(F, N_1T) \neq 0$. Then $a(F, T) \neq 0$.
\end{lemma}
\begin{proof} For each fixed $T$, we prove the statement by using induction on the number of primes (counted with multiplicity) dividing $N_1$. The statement is trivially true if $N_1 = 1$. Now let $N_1 >1$ and let $a(F, N_1T) \neq 0$. We need to show that $a(F, T) \neq 0$. Let $p$ be a prime dividing $N_1$ and suppose that $U(p) F = \lambda_p F$; such a $\lambda_p$ exists by our assumption on $F$. By~\eqref{upaction}, this means that $a(F, pS) = \lambda_p a(F,S)$ for all $S \in \Lambda_2$. Applying this fact for $S= (N_1/p)T$ and using the assumption $a(F, N_1T) \neq 0$, we deduce that $a(F, (N_1/p)T) \neq 0$. Now the induction hypothesis shows that $a(F, T) \neq 0$.

\end{proof}

\subsection{Non-vanishing of Fourier coefficients}\label{nonvanishingsec}
Recall that elements $S$ of $\Lambda_2$ are matrices of the form
$$
 S=\mat{a}{b/2}{b/2}{c},\qquad a,b,c\in\Z, \qquad a>0, \qquad \disc(S) := b^2 - 4ac < 0.
$$
%The quantity $\gcd(a,b,c)$ is called the content of $S$. Matrices of content $1$ are called primitive.
If $\gcd(a,b,c)=1$, then $S$ is called \emph{primitive}. If $\disc(S)$ is a fundamental discriminant, then $S$ is called \emph{fundamental}. Observe that if $S$ is fundamental, then it is automatically primitive. Observe also that if $\disc(S)$ is odd, then $S$ is fundamental if and only if $\disc(S)$ is squarefree.

 In an earlier work~\cite{sahafund} of the first author, it was shown that elements of $S_k^{(2)}(1)$ are uniquely determined by almost all of their fundamental Fourier coefficients. We now extend that result to elements of $S_k^{(2)}(N)$ under some assumptions as well as make it quantitative\footnote{Note however, that in the full-level case treated in~\cite{sahafund}, $k$ was allowed to be any integer while here we will restrict to $k$ even.}. In the theorem below, $\sq$ denotes the set of odd squarefree positive integers.

\begin{theorem} \label{th:nonvanfouriersiegel}Let $k>2$ be even and $N$ be a squarefree integer. Let $0 \ne F \in S_k^{(2)}(N)$ be an eigenfunction for the $U(p)$ operator for all primes $p$ dividing $N$. Then, for any $0< \delta < \frac58$, one has the lower bound
$$|\{0 < d < X , \ d \in \sq, \ a(F, S) \ne 0 \text{ for some }  S  \text{ with }  d = -\disc (S)   \}| \gg_{F, \delta} X^\delta. $$

\end{theorem}

\begin{remark} In particular, this implies that for $k, N$ as above, if $F \in S_k^{(2)}(N)$ is non-zero and an eigenfunction for the $U(p)$ operator for all primes $p$ dividing $N$, then there exist infinitely many fundamental $S$ such that $a(F,S) \ne 0$.
\end{remark}

\begin{proof} The proof is very similar to that of Theorem 1 of~\cite{sahafund}. Let $F\in S_k^{(2)}(N)$ be non-zero and an eigenfunction for the $U(p)$ operator for all primes $p$ dividing $N$. A result of Yamana~\cite[Thm. 2]{yamana09} tells us that there exists a \emph{primitive} matrix $T$ and an integer $N_1$ dividing $N$ such that $a(F, N_1T) \neq 0$. Now by Lemma~\ref{nonvanlemma}, it follows that $a(F, T) = 0$.

Since $T$ is primitive, we can write $T= \mat{n}{r/2}{r/2}{m}$ with $\gcd(m,r,n) = 1$ and $4nm> r^2$. By the main theorem of~\cite{iwanprime}, there exist infinitely many primes of the form $mx_0^2 + rx_0y_0 + ny_0^2$. We pick a prime $p$ such that $p \nmid N$ and $p = mx_0^2 + rx_0y_0 + ny_0^2$. Since this implies $\gcd(x_0, y_0) = 1$, we can find integers $x_1$, $y_1$ such that $A= \mat{y_1}{y_0}{x_1}{x_0} \in \SL_2(\Z).$ Let $T' =\,\T\!ATA$. Then $a(F,T) = a(F, T')$ and $T'$ is of the form $\mat{n'}{r'/2}{r'/2}{p}$.

This implies that there is a prime $p$ not dividing $N$ such that the Jacobi form $\phi_p$ in the expansion~\eqref{fjexpand} satisfies $\phi_p \neq 0$. Let us denote $$c(n,r) = a\left(F, \mat{n}{r/2}{r/2}{p}\right).$$ Then the Fourier expansion of $\phi_p$ is given by $$\phi_p(\tau, z) = \sum_{\substack{n,r \in \Z \\ 4np> r^2}} c(n,r) e(n \tau) e(r z).$$ By our assumption $c(n', r') \neq 0$, where $$T' = \mat{n'}{r'/2}{r'/2}{p}.$$
Now, let $$h(\tau) = \sum_{m=1}^\infty  c(m) e(m \tau).$$
where $$c(m) = \sum_{\substack{0 \le \mu \le 2p-1 \\ \mu^2 \equiv -m \pmod{4p}}} c\left((m+\mu^2)/4p, \mu \right).$$
By Theorem 4.8 of~\cite{manickram}, we know that $h \in S_{k-\frac{1}{2}}(4pN)$. Here $S_{k-\frac{1}{2}}(4pN)$ denotes the space of cusp forms of weight $k - \frac12$ for $\Gamma_0(4pN)$; for the basic definitions and properties of such spaces of half-integral forms, see for instance~\cite[Sect.\ 3.1]{sahafund}.

It is easy to see that $h(\tau)$ is not identically equal to 0. Indeed put $d_0 = 4 n' p - r'^2$. Then $c(d_0)$ equals $a\left(F, \mat{n'}{r'/2}{r'/2}{p}\right) + a\left(F, \mat{n'+p-r'}{p - r'/2}{p - r'/2}{p}\right)$, which is simply $2 a\left(F, \mat{n'}{r'/2}{r'/2}{p}\right)$ by~\eqref{fourierinvariance} and hence non-zero.

Now, by Theorem~\ref{hafintgen} below, it follows that $$|\{0 < d<X, \ d \in \sq, \ c(d) \ne 0\}|\gg_{h,\delta} X^\delta. $$ For any of these $d$, there exists a $\mu$ such that  $c\left(\frac{d+\mu^2}{4p},\mu \right) = a \left(F, \mat{\frac{d+\mu^2}{4p}}{\mu/2}{\mu/2}{p}\right)$ is not equal to zero. This completes the proof.

\end{proof}

\subsection{A result on half-integral weight cusp forms}\label{halfintegralsec}
The following theorem, which is a generalization of Theorem 2 of~\cite{sahafund}, was used in the proof above. We refer the reader to~\cite[Sect.\ 3.1]{sahafund} for the notations and definitions related to half-integral weight cusp forms.
\begin{theorem}\label{hafintgen}Let $N$ be a positive integer divisible by $4$ and $\chi : (\Z / N \Z)^\times \rightarrow \C^\times$ be a character. Write $\chi = \prod_{p | N} \chi_p$ and assume that the following conditions are satisfied:

\begin{enumerate}
\item $N$ is not divisible by 16, and if $N$ is divisible by $8$, then $\chi_2 =1$.

 \item   $N$ is not divisible by $p^3$ for any odd prime $p$,
\item If $p$ is an odd prime such that $p^2$ divides $N$, then $\chi_p \ne 1$.

\end{enumerate}
For some $k \ge 2$, let $f \in S_{k+\frac{1}{2}}(N, \chi)$ be non-zero with the Fourier expansion $f(z) = \sum_{n > 0} a(f, n)e(nz).$  Then, for any $0< \delta < \frac58$, one has the lower bound
$$|\{0 < d < X , \ d \in \sq, \ a(f,d) \ne 0  \}| \gg_{f, \delta} X^\delta. $$

\end{theorem}
The rest of this subsection is devoted to the proof of the above theorem. We start with the following key proposition.

\begin{proposition}\label{keyprop}Let $k \ge 2$, $N$ be a positive integer that is divisible by $4$ and $\chi$ a Dirichlet character $\ \bmod \ N$. Let $f \in S_{k + \frac12}(N, \chi)$, $f \ne 0$, and suppose that $a(f,n)$ equals 0 whenever $n$ and $N$ have a common prime factor. Then,  for any $0< \delta < \frac58$, one has the lower bound
$$|\{0 < d < X , \ d \in \sq, \ a(f,d) \ne 0  \}| \gg_{f, \delta} X^\delta. $$

\end{proposition}

\begin{proof} The qualitative version of this proposition, i.e., the assertion that there are infinitely many $d\in \sq$ such that $a(f,d) \neq 0$, is just Proposition 3.7 of~\cite{sahafund}. The proof of the quantitative version as stated here requires no new ingredients. Indeed, the proof there proceeded by showing that there exists an integer $M$ such that $$S(M,X;f) := \sum_{\substack{d\in \sq \\ (d, M) = 1}} |\tilde{a}(f, d)|^2 e^{-d/X}$$
satisfies $S(M,X;f) \gg_f X$. Here $\tilde{a}(f,n)$ denotes the ``normalized" Fourier coefficients, defined by
 $$\tilde{a}(f,n) = a(f,n)n^{ \frac14-\frac k2}.$$ Proposition~\ref{keyprop} now follows immediately from the well-known bound due to Bykovski{\u\i}~\cite{byko} that $$|\tilde{a}(f, m)|^2 \ll_{f,\epsilon} m^{\frac{3}{8} + \epsilon}.$$

\end{proof}

We prove Theorem~\ref{hafintgen}. Let $f \in S_{k+\frac{1}{2}}(N, \chi)$ be non-zero where $N, \chi$ satisfy the assumptions listed in the statement of the theorem. Let $2=p_1, p_2,\ldots,p_t$ be the distinct primes dividing $N$. For $1 \le i \le t$, let $S_i = \{p_1, \ldots, p_i\}$.  We will construct a sequence of forms $g_i$, $0 \le i \le t$, such that \begin{enumerate}

\item $g_0 = f$,
\item $g_i \ne 0$ for any $0 \le i \le t$,
\item For $1 \le i \le t$, $g_i \in S_{k+ \frac12}(N N_i, \chi \chi_i)$ where $N_i$ is not divisible by any prime lying outside $S_i$ and $\chi_i$ is a Dirichlet character whose conductor is not divisible by any prime outside $S_i$,
\item $a(g_i, n) = 0$ whenever $n$ is divisible by a prime in $S_i$,
 \item For some $1 \le i \le t$, suppose it is true that $$|\{0 < d < X , \ d \in \sq, \ a(g_i,d) \ne 0  \}| \gg_{g_i, N, \delta} X^\delta. $$ Then it is also true that $$|\{0 < d < X , \ d \in \sq, \ a(g_{i-1},d) \ne 0  \}| \gg_{g_{i-1}, N, \delta} X^\delta. $$
\end{enumerate}

 It is clear that the existence of such a sequence of forms, along with Proposition~\ref{keyprop} above, directly implies Theorem~\ref{hafintgen}. The proof of the fact that such forms $g_i$ exist follows exactly the argument in Section 3.5 of~\cite{sahafund}. Indeed, the only difference is that in~\cite{sahafund}, $N/4$ was assumed to be odd, while here we allow $N/4$ to be divisible by 2 (but not by 4) so long as $\chi_2 =1$. However, a careful look at the proof of Theorem 2 of~\cite{sahafund} shows that the only place where the assumption $N/4$ odd was used was in Section 3.4 in order to show that $g_1 \neq 0$. However if $g_1 = 0$, then $a(g_0, n) = 0$ unless $n$ is even.  By~\cite[Lemma 7]{Serre-Stark}, this implies that the conductor of $\epsilon_2$ divides $N/4$;  here $\epsilon_{2}$ is the quadratic character associated to the field $\Q(\sqrt{2}).$ Since this conductor is equal to 8, this means that $N$ is divisible by 32, a contradiction. This completes the proof of Theorem~\ref{hafintgen}.
\section{Yoshida lifts}\label{yoshidasec}
\subsection{Siegel cusp forms and representations}\label{siegelfromrepsec}
Below we will use a representation theoretic construction of certain elements of $S_k^{(2)}(N)$. In preparation, we will briefly explain the relationship between Siegel modular forms of degree $2$ and automorphic representations of $G=\GSp_4$. For the full modular group this was explained in \cite{asgsch}, and even though we will now require levels, we may still refer to this paper for some details. In the level case the precise correspondence between modular forms and representations is complicated, due to a lack of multiplicity one both locally and globally. However, all we will have to do is construct a cusp form from an irreducible, cuspidal, automorphic representation, and this direction of the correspondence is unproblematic.

Throughout let $\A$ be the ring of adeles of $\Q$. Let $\pi=\otimes\pi_v$ (restricted tensor product) be a cuspidal, automorphic representation of the adelized group $G(\A)$ with trivial central character. The only requirement on $\pi$ necessary for the construction of classical, holomorphic modular forms is that the archimedean component $\pi_\infty$ is a lowest weight representation $\mathcal{E}(l,l')$ with integers $l\geq l'>0$ in the notation of \cite[Sect.\ 2.3]{pitsch2}. If $l'\geq3$, then $\mathcal{E}(l,l')$ is a holomorphic discrete series representation with Harish-Chandra parameter $(l-1,l'-2)$, but $l'=1$ and $l'=2$ are also admissible. Let $K_\infty\cong U(2)$ be the standard maximal compact subgroup of ${\rm Sp}_4(\R)$, and let $(\tau_{l,l'},W_{l,l'})$ be a model for the minimal $K_\infty$-type of $\mathcal{E}_{l,l'}$. Then $\dim W_{l,l'}=l-l'+1$. Up to multiples, the representation $\mathcal{E}(l,l')$ contains a unique vector of weight $(l,l')$ (see \cite[Sect.\ 2.2]{pitsch2}); it corresponds to a highest weight vector $w_1$ in $W_{l,l'}$. In the given model $\pi_\infty$ (which is arbitrary) we fix a non-zero such vector and denote it by $f_\infty$.

As for finite places, if $p$ is a prime such that $\pi_p$ is an unramified representation, let $f_p$ be a non-zero vector in $\pi_p$ such that $f_p$ is invariant under $G(\Z_p)$. For other primes, fix any non-zero vector $f_p$ in $\pi_p$, and let $K_p$ be any compact and open subgroup of $G(\Q_p)$ such that $f_p$ is invariant under $K_p$; for example, $K_p$ could be a principal congruence subgroup of $G(\Z_p)$ of high enough level. Then
$$
 \Gamma=G(\Q)^+\cap\prod_{p<\infty}K_p,
$$
where the superscript ``$+$'' denotes elements with positive multiplier, is a discrete subgroup of ${\rm Sp}_4(\R)$.

Now let $\Phi$ be the vector in the space of $\pi$ corresponding to the pure tensor $\otimes f_v$ in $\otimes\pi_v$. Then $\Phi$ is a $\C$-valued function on $G(\A)$ which is left-invariant under $G(\Q)$ and right-invariant under $\prod_{p<\infty}K_p$. We would like to construct from $\Phi$ a function taking values in the contragredient representation $W_{l,l'}^\vee$. We claim that there exists a unique function $L:\:G(\A)\rightarrow W_{l,l'}^\vee$ such that
\begin{equation}\label{tildePhidefeq}
 \Phi(g)=L(g)(w_1)\qquad\text{for all }g\in G(\A).
\end{equation}
Indeed, let $w_1,\ldots,w_n$ be a basis of $W_{l,l'}$ such that $w_2,\ldots,w_n$ have weights different from $w_1$, and let $L_1,\ldots,L_n$ be a basis of $W_{l,l'}^\vee$. Then, by the Peter-Weyl theorem, there exist uniquely determined complex numbers $c_{ij}(g)$ such that
$$
 \Phi(gh)=\sum_{i,j=1}^nc_{ij}(g)L_i(\tau_{l,l'}(h)w_j)\qquad\text{for all }h\in K_\infty.
$$
Since $\Phi$ has weight $(l,l')$, it follows that $c_{ij}=0$ for $j\neq1$. Hence (\ref{tildePhidefeq}) holds with $L(g)=\sum_ic_{i1}(g)L_i$. The uniqueness of $L$ follows from the construction.

Observe that, by construction, $\Phi(gh)=(\tau_{l,l'}^\vee(h^{-1})L(g))w_1$ for all $h\in K_\infty$, and $L$ is characterized by this property. This implies that
\begin{equation}\label{tildePhirhoeq}
 L(gh)=\tau_{l,l'}^\vee(h^{-1})L(g)\qquad\text{for all }g\in G(\A)\text{ and }h\in K_\infty.
\end{equation}
Furthermore, $L$ is left-invariant under $G(\Q)$ and right-invariant under $\prod_{p<\infty}K_p$.

We can now construct a modular form $F$ on the Siegel upper half space $\mathbb{H}_2$ taking values in $W_{l,l'}^\vee$. First we extend $\tau_{l,l'}^\vee$, which is a representation of $U(2)$, to a representation of $\GL_2(\C)$; by the unitary trick, this can be done in exactly one way. It is easy to verify that this extension has highest weight $(l,l')$ in the sense of \cite[Appendix to I.6]{Fr1991}. We will write $\rho_{l,l'}$ for this extension. For $Z$ in $\mathbb{H}_2$, let $g$ be an element of ${\rm Sp}_4(\R)$ such that $Z=g\langle i_2\rangle$, and set
$$
 F(Z)=\rho_{l,l'}(Ci_2+D)L(g),\qquad\text{where }g=\mat{A}{B}{C}{D}.
$$
Then $F$ is a well-defined holomorphic function on $\mathbb{H}_2$ with values in the space of $\rho_{l,l'}$. It satisfies
\begin{equation}\label{Frhollpeq}
 F(\gamma\langle Z\rangle)=\rho_{l,l'}(CZ+D)F(Z),\qquad\text{for }\gamma=\mat{A}{B}{C}{D}\in\Gamma.
\end{equation}
Hence, $F$ is a vector-valued modular form of type $\rho_{l,l'}$ with respect to $\Gamma$, in the sense of \cite{Fr1991}. It is a cusp form, and it is an eigenform of the local Hecke algebras $\mathcal{H}_p$ at each place $p$ where $\pi_p$ is unramified. It is scalar-valued if and only if $l=l'$.

In our application below we will have a situation where each $K_p$ can be chosen to be a Siegel congruence subgroup, i.e., a group of type
\begin{equation}\label{Gamma0localdefeq}
\Gamma_{0,p}(M) := \left\{\begin{pmatrix}A&B\\ C&D \end{pmatrix} \in G(\Z_p)\;|\;C \equiv 0 \pmod{M\Z_p}\right\}.
\end{equation}
These are, of course, the local analogues of the groups defined in (\ref{Gamma0defeq}). If $K_p=\Gamma_{0,p}(p^{m_p})$ for all $p$, and if $l=l'=k$, then the resulting $F$ will be an element of the space $S_k^{(2)}(N)$ defined earlier, where $N=\prod_pp^{m_p}$.

Provided that the multiplier maps from the local groups $K_p$ to $\Z_p^\times$ are all surjective, the above procedure can be reversed (see \cite[Sect.\ 4.5]{asgsch}), and one can reconstruct the automorphic representation $\pi$ from the modular form $F$. In general, starting from an arbitrary cusp form $F$ which is an eigenform for almost all local Hecke algebras, it is unclear whether $F$ generates an \emph{irreducible} automorphic representation.
\subsection{Representation-theoretic Yoshida liftings}\label{yoshidarepsec}
In the language of automorphic representations, the Yoshida lifting is a certain case of Langlands functoriality. We will first make some comments on  dual groups and $L$-packets, and then explain how the Yoshida lifting can be constructed using the theta correspondence. In the next section we will use this group theoretic lifting to construct holomorphic Siegel modular forms.

We fix a totally real number field $F$. The Yoshida lifting comes from the embedding of dual groups
\begin{align}\label{dualgroupmorphismeq}
 \{(g_1,g_2)\in\GL_2(\C)\times\GL_2(\C)\:|\:\det(g_1)=\det(g_2)\}&\longrightarrow\GSp_4(\C),\\
 (\mat{a}{b}{c}{d},\mat{a'}{b'}{c'}{d'})&\longrightarrow\left(\begin{matrix}a&&b\\&a'&&b'\\c&&d\\&c'&&d'\end{matrix}\right).\nonumber
\end{align}
The principle of functoriality predicts that for a pair $\tau_1,\tau_2$ of automorphic representations of $\GL_2(\A)$ with the same central character there exists an $L$-packet $\Pi(\tau_1,\tau_2)$ of automorphic representations of $G(\A)$ such that
\begin{equation}\label{basicLrelationeq}
 L(s,\Pi(\tau_1,\tau_2))=L(s,\tau_1)L(s,\tau_2).
\end{equation}
The trivial central character version of the Yoshida lifting comes from the embedding of dual groups $\SL_2(\C)\times\SL_2(\C)\rightarrow{\rm Sp}_4(\C)$ given by the same formula, and predicts that for a pair $\tau_1,\tau_2$ of automorphic representations of $\PGL_2(\A)$ there exists an $L$-packet $\Pi(\tau_1,\tau_2)$ of automorphic representations of $\PGSp_4(\A)$ such that (\ref{basicLrelationeq}) holds.

Let us assume that $\tau_1=\otimes\tau_{1,v}$ and $\tau_2=\otimes\tau_{2,v}$ with irreducible, admissible, \emph{tempered} representations $\tau_{1,v},\tau_{2,v}$ of $\GL_2(F_v)$; this is all we need for our application. By the results of \cite{gantakGSp4} or the construction in \cite{rob2001}, the local $L$-packets resulting from the morphism (\ref{dualgroupmorphismeq}) have one or two elements. A packet has two elements precisely if $\tau_{1,v}$ and $\tau_{2,v}$ are both discrete series representations. In this case
$$
 \Pi(\tau_{1,v},\tau_{2,v})=\{\Pi_v^{\rm gen},\Pi_v^{\rm ng}\},
$$
where $\Pi_v^{\rm gen}$ is a generic representation and $\Pi_v^{\rm ng}$ is a non-generic representation. By definition, the global $L$-packet $\Pi(\tau_1,\tau_2)$ consists of all representations $\Pi=\otimes\Pi_v$ where $\Pi_v$ lies in the local packet $\Pi(\tau_{1,v},\tau_{2,v})$. Hence, if $T$ denotes the set of places where both $\tau_{1,v}$ and $\tau_{2,v}$ are square-integrable, then the number of irreducible admissible representations of $G(\A)$ in $\Pi(\tau_1,\tau_2)$ is $2^{\#T}$.

Arthur's multiplicity formula makes a precise prediction on which elements of the global packet occur in the discrete automorphic spectrum. Given $\Pi=\otimes\Pi_v$ in $\Pi(\tau_1,\tau_2)$, let $T^{\rm ng}$ be the set of places where $\Pi_v$ is non-generic (this can only happen at places where the local packet has two elements). Then the prediction is that $\Pi$ occurs in the discrete spectrum if and only if $\#T^{\rm ng}$ is \emph{even}. Hence, the number of discrete elements in the $L$-packet is
$$
 \#\Pi(\tau_1,\tau_2)_{\rm disc}=\begin{cases}
  1&\text{if }T=\emptyset,\\
  2^{\#T-1}&\text{if }T\neq\emptyset.
                                 \end{cases}
$$
Thus, the global packet is \emph{finite} and \emph{unstable}. The prediction of Arthur's multiplicity formula in this situation has been proven in \cite[Thm.\ 8.6 (2)]{rob2001}. It turns out that in fact the discretely occurring $\Pi$ are cuspidal, automorphic representations.

The construction of the local and global packets in \cite{gantakGSp4} and \cite{rob2001} uses the theta correspondence (with similitudes) between $G=\GSp_4$ and various orthogonal groups of four-dimensional quadratic spaces. If $D_v$ is a (possibly split) quaternion algebra over $F_v$, considered as a quadratic space with the reduced norm, then it is well known that there is an exact sequence
\begin{equation}\label{GSOexacteq}
 1\longrightarrow F_v^\times\longrightarrow D_v^\times\times D_v^\times\longrightarrow{\rm GSO}(D_v)\longrightarrow1.
\end{equation}
Thus, representations of ${\rm GSO}(D_v)$ can be identified with pairs of representations of $D_v^\times$ with the same central character. Each such pair gives then rise to a representation of $G(F_v)$ via the theta correspondence. More precisely, one first induces from ${\rm GSO}(D_v)$ to ${\rm GO}(D_v)$; if this induction is irreducible, it participates in the theta correspondence with $G(F_v)$, and if it is not irreducible, there is a unique irreducible component that participates in the theta correspondence with $G(F_v)$. See Sect.\ 3 of \cite{gantakGSp4} for more information on the relationship between ${\rm GSO}(D_v)$ and ${\rm GO}(D_v)$.

The construction of the local packets $\Pi(\tau_{1,v},\tau_{2,v})$ above is now as follows. First, let $D_v=M_2(F_v)$ be the split quaternion algebra, so that $D_v^\times=\GL_2(F_v)$. Then, via the theta correspondence, the pair $\tau_{1,v},\tau_{2,v}$ gives rise to an irreducible, admissible representation of $G(F_v)$, which is the \emph{generic} member $\Pi^{\rm gen}_v$ in the local packet. Next let $D_v$ be the unique division quaternion algebra over $F_v$. If $\tau_{1,v}$ and $\tau_{2,v}$ are both square integrable, we transfer these representations to $D_v^\times$ via the Jacquet-Langlands correspondence. Using again the theta correspondence, the pair of representations thus obtained gives rise to another representation of $G(F_v)$, which is the \emph{non-generic} member $\Pi_v^{\rm ng}$ of the local packet.

In the global case, let $T$ be as above and let $T^{\rm ng}$ be a subset of $T$ of even cardinality. Let $D$ be the global quaternion algebra over $F$ which is non-split at exactly the places in $T^{\rm ng}$. We use the Jacquet-Langlands lifting to transfer $\tau_1$ and $\tau_2$ to automorphic representations $\tau_1'$ and $\tau_2'$ of $D_\A^\times$. By the global analogue of the exact sequence (\ref{GSOexacteq}), we obtain an automorphic representation of the group ${\rm GSO}(D_\A)$. It was proved in \cite{rob2001} that the global theta lifting of this representation to $G(\A)$ is non-vanishing (again, one should first transition to the global group ${\rm GO}(D_\A)$; see \cite[Sect.\ 7 and proof of Thm.\ 8.5]{rob2001}). It follows from the compatibility of the local and global theta correspondence that this global lifting is the element $\Pi=\otimes\Pi_v$ in the global packet $\Pi(\tau_1,\tau_2)$ for which $\Pi_v$ is non-generic exactly at the places in $T^{\rm ng}$.

We close this section by giving a more explicit description of the non-archime\-dean local packets with two elements. To ease the notation, let us omit the subscript $v$ from the local field $F$. The notation we use for irreducible, admissible representations of $G(F)$ goes back to \cite{sallytadic}. The classification into types I, II, etc.\ is taken from \cite{NF}. The $L$-parameters listed in Table A.7 of \cite{NF} coincide with those defined in \cite{gantakGSp4}. In the following table $\sigma$ stands for an arbitrary character of $F^\times$, and $\xi$ denotes a non-trivial quadratic character. The Steinberg representation of $\GL_2(F)$ is denoted by ${\rm St}_{\GL(2)}$. The symbols $\pi$, $\pi_1$ and $\pi_2$ stand for supercuspidal representations of $\GL_2(F)$, and $\omega_\pi$ denotes the central character of $\pi$. Finally, $\nu$ denotes the normalized absolute value on $F$.
\begin{equation}\label{padicpacketseq}
 \begin{array}{cclcc}
   \tau_1&\tau_2&\multicolumn{2}{c}{\Pi(\tau_1,\tau_2)}&\text{type}\\
  \toprule
   \sigma\St_{\GL(2)}&\xi\sigma\St_{\GL(2)}&\Pi^{\rm gen}&\delta([\xi,\nu\xi],\nu^{-1/2}\sigma)&{\rm Va}\\
   &&\Pi^{\rm ng}&\multicolumn{2}{c}{\text{non-generic supercuspidal}}\\
   \sigma\St_{\GL(2)}&\sigma\St_{\GL(2)}&\Pi^{\rm gen}&\tau(S,\nu^{-1/2}\sigma)&{\rm VIa}\\
   &&\Pi^{\rm ng}&\tau(T,\nu^{-1/2}\sigma)&{\rm VIb}\\
   \pi&\pi&\Pi^{\rm gen}&\tau(S,\pi)&{\rm VIIIa}\\
   &&\Pi^{\rm ng}&\tau(T,\pi)&{\rm VIIIb}\\
   \sigma\St_{\GL(2)}&\sigma\pi\;(\omega_\pi=1)&\Pi^{\rm gen}&\delta(\nu^{1/2}\pi,\nu^{-1/2}\sigma)&{\rm XIa}\\
   &&\Pi^{\rm ng}&\multicolumn{2}{c}{\text{non-generic supercuspidal}}\\
   \pi_1&\pi_2\;(\not\cong\pi_1)&\Pi^{\rm gen}&\multicolumn{2}{c}{\text{generic supercuspidal}}\\
   &&\Pi^{\rm ng}&\multicolumn{2}{c}{\text{non-generic supercuspidal}}\\
 \end{array}
\end{equation}
We remark that these packets, and many more, appear also in \cite{robjjl}.
\subsection{Classical Yoshida liftings}
In view of the procedure explained in Sect.\ \ref{siegelfromrepsec}, the representation theoretic construction outlined in the previous section may be used to construct holomorphic Siegel modular forms. We will now work over the number field $\Q$. For simplicity, we will consider the trivial central character version of the Yoshida lifting. For $i=1,2$ let $\tau_i=\otimes\tau_{i,v}$ be a cuspidal, automorphic representation of $\GL_2(\A)$ corresponding to a primitive cuspform $f_i$ of (even) weight $k_i$ and level $N_i$. Further, we will make the assumption that $N_1$ and $N_2$ are squarefree, since complete local information is currently only available in this case (however, it is possible to construct holomorphic Yoshida liftings in somewhat greater generality). We will also assume that $k_1 \ge k_2$. Since the temperedness hypothesis is satisfied, we obtain a global $L$-packet $\Pi(\tau_1,\tau_2)$ as in the previous section.

To understand the local packet at the archimedean place, let $W_\R=\C^\times\sqcup j\C^\times$ be the real Weil group, as in \cite{knapparch}. For an odd, positive integer $l$ let $\varphi_l$ be the two-dimensional, irreducible representation of $W_\R$ given by
\begin{equation}\label{GL2archWeil4}
 \C^\times\ni re^{i\theta} \longmapsto\mat{e^{il\theta}}{}{}{e^{-il\theta}},\quad j\longmapsto\mat{}{-1}{1}{}.
\end{equation}
By the local Langlands correspondence, $\varphi_l$ is the parameter of a discrete series representation of $\PGL_2(\R)$ with minimal weight $l+1$. Hence, the archimedean parameter of $\tau_{i,\infty}$ is $\varphi_{k_i-1}$, for $i=1,2$. Composing with the dual group morphism (\ref{dualgroupmorphismeq}), we obtain the parameter $\varphi_{k_1-1}\oplus\varphi_{k_2-1}$ (as a representation of $W_\R$). If $k_1\geq k_2+2$, it corresponds to a two-element packet of discrete series representations of $\PGSp_4(\R)$ with Harish-Chandra parameter
\begin{align*}
 (\lambda_1,\lambda_2)&=\Big(\frac{(k_1-1)+(k_2-1)}2,\:\frac{(k_1-1)-(k_2-1)}2\Big)\\
 &=\Big(\frac{k_1+k_2-2}2,\:\frac{k_1-k_2}2\Big).
\end{align*}
In order to obtain holomorphic modular forms, we need to choose the holomorphic element in the $L$-packet. In the notation of Sect.\ \ref{siegelfromrepsec}, this is the lowest weight representation $\mathcal{E}(l,l')$ with
\begin{equation}\label{minKk1k2eq}
 (l,l')=\Big(\frac{k_1+k_2}2,\:\frac{k_1-k_2+4}2\Big)
\end{equation}
(the $(1,2)$-shift between the Harish-Chandra parameter and the minimal $K$-type is half the sum of the positive non-compact roots minus half the sum of the positive compact roots). Hence, the element $\Pi=\otimes\Pi_v$ in the global packet $\Pi(\tau_1,\tau_2)$ we are going to construct will have this lowest weight representation as its archimedean component $\Pi_\infty$. If $k_1=k_2$, these considerations remain true except that $\mathcal{E}(l,l')$ will be a limit of discrete series representation.

Now $\Pi_\infty$ is known to be the \emph{non-generic} member $\Pi^{\rm ng}_\infty$ of the archimedean packet. Therefore, in order to satisfy the parity condition coming from Arthur's multiplicity formula, we require an \emph{odd} number of (finite) primes $p$ such that $\Pi_p$ is non-generic. Under our assumption that $N_i$ is squarefree, the local component $\tau_{i,p}$ is square-integrable if and only if $p|N_i$. Hence, the parity condition can be satisfied if and only if $M:=\gcd(N_1,N_2)>1$. We will make this assumption.

It is well known (and easy to verify) that, for $p|N_i$, the local component $\tau_{i,p}$ is an unramified twist of the Steinberg representation. More precisely, if the Atkin-Lehner eigenvalue of $f_i$ at $p$ is $-1$, then $\tau_{i,p}=\St_{\GL(2)}$, and otherwise $\tau_{i,p}=\xi\St_{\GL(2)}$, where $\xi$ is the non-trivial, quadratic, unramified character of $\Q_p^\times$. The local packets for places $p|M$ can now be read off table (\ref{padicpacketseq}). For places $p\nmid M$ but $p|N$, where $N={\rm lcm}(N_1,N_2)$, the local packets have one element and can be read off \cite[Table A.7]{NF}. The following table summarizes all possibilities of local packets for $p|N$. The character $\sigma$ in the table is quadratic and unramified, but allowed to be trivial. Since we would like to construct modular forms with respect to $\Gamma_{0,p}(N)$, we have indicated in the last column the dimension of fixed vectors under the local Siegel congruence subgroup $\Gamma_{0,p}(p)$ defined in (\ref{Gamma0localdefeq}). This information about fixed vectors is taken from \cite[Table A.15]{NF}; note that Va$^*$, being a supercuspidal representation, has no Iwahori fixed vectors.
\begin{equation}\label{padicpacketssquarefreeeq}
 \begin{array}{cclccc}
   \tau_1&\tau_2&\multicolumn{2}{c}{\Pi(\tau_1,\tau_2)}&\text{type}&\dim\\
  \toprule
   \sigma\St_{\GL(2)}&\xi\sigma\St_{\GL(2)}&\Pi^{\rm gen}&\delta([\xi,\nu\xi],\nu^{-1/2}\sigma)&{\rm Va}&0\\
   &&\Pi^{\rm ng}&\delta^*([\xi,\nu\xi],\nu^{-1/2}\sigma)&{\rm Va}^*&0\\
   \sigma\St_{\GL(2)}&\sigma\St_{\GL(2)}&\Pi^{\rm gen}&\tau(S,\nu^{-1/2}\sigma)&{\rm VIa}&1\\
   &&\Pi^{\rm ng}&\tau(T,\nu^{-1/2}\sigma)&{\rm VIb}&1\\
   \sigma\St_{\GL(2)}&\chi\times\chi^{-1}\;\text{(unram.)}&\multicolumn{2}{c}{ \qquad \quad \sigma\chi^{-1}\St_{\GL(2)}\rtimes\chi}&{\rm IIa}&1
 \end{array}
\end{equation}
We see from the last column that, in order to construct modular forms with respect to $\Gamma_0(N)$, we need to completely avoid the packet $\{{\rm Va},{\rm Va}^*\}$. In other words, the Atkin-Lehner eigenvalues of $f_1$ and $f_2$ need to coincide for all $p|M$, an assumption we will make from now on.

Under this assumption we have either a VIa or VIb type representation at places $p|M$, and since either one of these representations contains a $\Gamma_{0,p}(p)$ fixed vector, we can make an arbitrary choice. As pointed out above, the only constraint is that the non-generic VIb has to appear an odd number of times. By the general procedure outlined in Sect.\ \ref{siegelfromrepsec} of constructing vector-valued modular forms from automorphic representations, we now obtain the following result. Even though it is not necessary for our applications further below, we have included a statement about Atkin-Lehner eigenvalues for completeness; see \cite[Sect.\ 3.2]{sch} for the definition of Atkin-Lehner involutions in the case of Siegel modular forms.

\begin{proposition}\label{classicalyoshidaprop}
 Let $k_1$ and $k_2$ be even, positive integers with $k_1\geq k_2$. Let $N_1$, $N_2$ be two positive, squarefree integers such that $M = \gcd(N_1, N_2)>1$. Let $f$ be a classical newform of weight $k_1$ and level $N_1$ and $g$ be a classical newform of weight $k_2$ and level $N_2$, such that $f$ and $g$ are not multiples of each other. Assume that for all primes $p$ dividing $M$ the Atkin-Lehner eigenvalues of $f$ and $g$ coincide. Put $N = \mathrm{lcm}(N_1, N_2)$. Then for any divisor $M_1$ of $M$ with an \emph{odd} number of prime factors, there exists a non-zero holomorphic Siegel cusp form $F_{f,g} = F_{f,g;M_1}$ with the following properties.
 \begin{enumerate}
  \item $F_{f,g}$ is a modular form with respect to $\Gamma_0(N)$ of type $\rho_{l,l'}$ (see (\ref{Frhollpeq})), where $(l,l')$ is as in (\ref{minKk1k2eq}).
  \item $F_{f,g}$ is an eigenfunction of the local Hecke algebra at all places $p\nmid N$, and generates an irreducible cuspidal representation $\Pi_{f,g}$ of $\GSp_4(\A)$.
  \item $F_{f,g}$ is an eigenfunction of the operator $U(p)$ for all $p|N$.
  \item For $p|N$, let $\epsilon_p$ be the Atkin-Lehner eigenvalue of $F_{f,g}$ at $p$, and let $\delta_p$ be the Atkin-Lehner eigenvalue of $f$ (if $p|N_1$) or $g$ (if $p|N_2$). Then, for all $p|N$,
   $$
    \epsilon_p=\begin{cases}
                \delta_p&\text{if }p\nmid M_1,\\
                -\delta_p&\text{if }p|M_1.
               \end{cases}
   $$
  \item There is an equality of (complete) Langlands $L$-functions
  $$
   L(s,\Pi_{f,g}) = L(s,\pi_f)L(s,\pi_g),
  $$
  where $\pi_f$ and $\pi_g$ are the cuspidal representations of $\GL_2(\A)$ attached to $f$ and $g$.
  \item Let $D$ be the definite quaternion algebra over $\Q$ ramified exactly at (infinity and) the primes dividing $M_1$. Let $\pi'_f$ (resp.\ $\pi'_g$) be the Jacquet-Langlands transfer of $\pi_f$ (resp.\ $\pi_g$) to $D^\times_\A$. Then $\Pi_{f,g}$ is the global theta lift from $(D^\times_\A \times  D^\times_\A)/\A^\times\cong{\rm GSO}(D_\A) $ to $\GSp_4(\A)$ of the representation $\pi'_f \boxtimes \pi'_g$.
 \end{enumerate}
\end{proposition}
\begin{proof}
 All statements except iii) and iv) follow from the construction explained in Sect.\ \ref{yoshidarepsec}. To prove iii), note that the operator $U(p)$ defined in (\ref{upaction}) corresponds to a certain element in the local Hecke algebra at $p$ consisting of left and right $\Gamma_{0,p}(p)$ invariant functions; see for example the appendix to \cite{boch-up}. Since, by (\ref{padicpacketssquarefreeeq}), the local space of $\Gamma_{0,p}(p)$ fixed vectors is one-dimensional in each case, this Hecke algebra acts via scalars on the one-dimensional spaces. In particular, $F_{f,g}$ is a $U(p)$ eigenvector. Finally, iv) can be deduced from (\ref{padicpacketssquarefreeeq}) and the Atkin-Lehner eigenvalues given in \cite[Table A.15]{NF}.
\end{proof}

\begin{remark}
 In our application below we will set $k_1=2k$ for a positive integer $k$ and $k_2=2$. In this case $F_{f,g}\in S_{k+1}^{(2)}(N)$.
\end{remark}

\begin{remark}
 The cusp forms $F_{f,g} = F_{f,g;M_1}$ constructed in the proposition are known as Yoshida lifts. The theory was initiated in \cite{yosh1980} using a ``semi-classical'' language. The non-vanishing problem for Yoshida's construction was solved in \cite{bocsch1991} for the scalar-valued case, and in \cite{bocsch1997} for the vector-valued case. While the Siegel cusp forms constructed in Proposition \ref{classicalyoshidaprop} and in \cite{bocsch1997} are the same, the representation theoretic approach is slightly better suited for our purposes. One reason is that the $F_{f,g}$ from Proposition \ref{classicalyoshidaprop} automatically generate an irreducible, cuspidal representation\footnote{R.\ Schulze-Pillot has pointed out to the authors that it can be shown using results of Moeglin \cite{Moeglin1997} that the Yoshida liftings of \cite{bocsch1997} indeed generate an irreducible cuspidal representation.} of $\GSp_4(\A)$.
\end{remark}
\begin{remark}
 In \cite{yosh1980} Yoshida also considers a construction of Siegel modular forms from Hilbert modular forms. For a thorough representation-theoretic treatment of this lifting, see \cite{rob2001} and \cite{robjjl}. The local data given in \cite{robjjl} shows that the resulting modular forms cannot be with respect to a congruence subgroup $\Gamma^{(2)}_0(N)$ with square-free $N$. The same is true for the imaginary-quadratic version of this lifting considered in \cite{harsoutay}, since the non-archimedean situation is identical. Since our Theorem \ref{th:nonvanfouriersiegel} is for square-free levels only, it does not apply to these kinds of Yoshida liftings.
\end{remark}
\begin{remark}
 For a positive integer $N$ let $\Gamma^{\rm para}(N)$ be the paramodular group of level $N$, as defined in \cite{robjjl}. It is not possible to construct holomorphic Yoshida lifts with respect to $\Gamma^{\rm para}(N)$, for any $N$. The reason is that, as pointed out above, holomorphy forces at least one of the finite components in $\Pi=\otimes\Pi_v$ to be one of the non-generic representations occurring in table (\ref{padicpacketseq}). By Theorem 3.4.3 of \cite{NF}, these non-generic representations have no paramodular vectors.
\end{remark}

\section{Bessel periods and  $L$-values}\label{besselsec}
\subsection{Bessel periods}\label{s:besselperiod}\nopagebreak
Let $S = \begin{pmatrix}
  a & b/2\\
  b/2 & c\\
\end{pmatrix} \in M_2(\Q)$ be a symmetric matrix. Put $d = 4ac-b^2$ and define the element
$$
\xi = \xi_S = \begin{pmatrix}
b/2 & c\\
-a & -b/2\\
\end{pmatrix}.
$$
Note that $$\xi^2 = \begin{pmatrix} -\frac{d}{4} &\\&-\frac{d}{4} \end{pmatrix}.$$ Let $K$ denote the subfield $\Q(\sqrt{-d})$ of $\C$.
We always identify $\Q(\xi)$ with $K$ via
\begin{equation}\label{e:L}
\Q(\xi)\ni x + y\xi \mapsto x +
y\frac{\sqrt{-d}}{2} \in \field, \ x,y\in \Q.
\end{equation}

We define a subgroup $T =T_S$ of $\GL(2)$ by
\begin{equation}
T(\Q) = \{g \in \GL(2,\Q)\;|\; \T{g}Sg =\det(g)S\}.
\end{equation}
It is not hard to verify that $T(\Q) = \Q(\xi)^\times$. We identify
$T(\Q)$ with $\field^\times$ via~\eqref{e:L}.
We can consider $T$ as a subgroup of $G$ via
\begin{equation}\label{embedding}
T \ni g \longmapsto
\begin{pmatrix}
g & 0\\
0 & \det(g)\ \T{g^{-1}}
\end{pmatrix} \in G.
\end{equation}
Let us denote by $U$ the subgroup of $G$ defined by
$$
U = \{u(X) =
\begin{pmatrix}
1_2 & X\\
0 & 1_2\\
\end{pmatrix}\;|\;\T{X} = X\}.
$$
Let $R$ be the subgroup of $G$ defined by $R=TU$.

Recall that $\A$ denotes the ring of adeles of $\Q$. Let $\psi = \prod_v\psi_v$ be a character of $\A$ such that
\begin{itemize}
\item The conductor of $\psi_p$ is $\Z_p$ for all (finite) primes $p$,
\item $\psi_\infty(x) = e(x),$ for $x \in \R$,
\item $\psi|_\Q =1.$
\end{itemize} We define the
character $\theta = \theta_S$ on $U(\A)$ by $\theta(u(X))=
\psi(\Tr(SX))$. Let $\chi$ be a character of $T(\A) / T(\Q)$ such
that $\chi |_{\A^\times}= 1$. Via~\eqref{e:L} we can think of
$\chi$ as a character of $\field^\times(\A)/\field^\times$ such
that $\chi |_{\A^\times} = 1$. Denote by $\chi \otimes \theta$ the
character of $R(\A)$ defined by $(\chi \otimes \theta)(tu) =
\chi(t)\theta(u)$ for $t\in
T(\A)$ and $u\in U(\A).$

 Let $\mathcal{A}_0(G(\Q)\bs G(\A), 1)$ denote the space of cusp forms on $G(\A)$ with trivial central character; thus any $\Phi \in \mathcal{A}_0(G(\Q)\bs G(\A), 1)$ can be written as a finite sum of vectors in irreducible cuspidal representations of $G(\A)$.

 For $\Phi \in \mathcal{A}_0(G(\Q)\bs G(\A), 1)$, we define the Bessel period $B(\Phi) = B_{S, \chi, \psi}(\Phi) $
by
\begin{equation}\label{defbessel}
  B(\Phi) =
  \int_{\A^\times R(\Q)\bs R(\A)} (\chi \otimes \theta)(r)^{-1}\Phi(r)\,dr.
\end{equation}
\par

\subsection{Relation with Fourier coefficients}
Now, let $d$ be a positive integer such that $-d$ is
the discriminant of the imaginary quadratic field
$\Q(\sqrt{-d})$ and define
\begin{equation}\label{defsd}
 S = S(-d) = \begin{cases} \begin{pmatrix}
  \frac{d}{4} & 0\\
 0 & 1\\\end{pmatrix} & \text{ if } d\equiv 0\pmod{4}, \\[4ex]
 \begin{pmatrix} \frac{1+d}{4} & \frac12\\\frac12 & 1\\
 \end{pmatrix} & \text{ if } d\equiv 3\pmod{4}.\end{cases}
\end{equation}
\par
 Let $K = \Q(\sqrt{-d})$ and define the group $R$ as in the previous subsection. Let $N$ be a positive integer and let $F$ be an element of $S_k^{(2)}(N)$ with the Fourier expansion \begin{equation}\label{siegelfourierexpansion2}F(Z)
=\sum_{T \in \Lambda_2} a(F, T) e(\Tr(TZ)).\end{equation}
We define the adelization $\Phi_F$ of $F$ to be the function on $G(\A)$ given by
\begin{equation}\label{adelization}
\Phi_F(\gamma h_\infty k_0) =
  \mu_2(h_\infty)^k \det (J(h_\infty,
  i_2))^{-k}F(h_\infty\langle i_2\rangle),
\end{equation}
where $\gamma \in G(\Q)$, $h_\infty \in G(\R)^+$ and
$$
k_0 \in \prod_{ p< \infty} \Gamma_{0,p}(N),
$$
where the group
$
\Gamma_{0,p}(N)$ defined in~\eqref{Gamma0localdefeq}
is the local analogue of $\Gamma_0(N)$.

It is not hard to see that $\Phi_F \in \mathcal{A}_0(G(\Q)\bs G(\A), 1)$. For a symmetric matrix $T\in M_2^{\rm sym}(\Q)$, define
\begin{equation}\label{fouriercoefficientdefeq}
 (\Phi_F)_{T}(g)=\int\limits_{M_2^{\rm sym}(\Q)\backslash M_2^{\rm sym}(\A)}\psi^{-1}({\rm tr}(TX))\Phi_F(\mat{1}{X}{}{1}g)\,dX,\qquad g\in G(\A).
\end{equation}

\begin{lemma}\label{globalfourierlemma}
 Let $F\in S_k^{(2)}(N)$  and $\Phi_F$ as in (\ref{adelization}). Then, for all $g\in G(\R)$,
 \begin{equation}\label{globalfourierlemmaeq1}
  (\Phi_F)_T(g)= \begin{cases}\mu_2(g)^k\,\det J(g,i_2)^{-k}a(F,T)e({\rm tr}(TZ)) & \text{ if } g\in G(\R)^+ , \\ 0 & \text{ if } g\in G(\R)^-, \end{cases}
 \end{equation}
 where $Z=g\langle i_2\rangle$.
\end{lemma}
\begin{proof}This is a standard calculation.\end{proof}
\begin{remark}A version of this lemma holds more generally for Siegel modular forms with respect to any congruence subgroup.

\end{remark}
Recall that $\Cl_K$ denotes the ideal
class group of $\field$. Let $(t_c)$, $c\in \Cl_K$, be coset
representatives such that
\begin{equation}\label{e:tcosetca1}
T(\A) = \bigsqcup_{c}\,t_cT(\Q)T(\R)\prod_{p<\infty}
(T(\Q_p) \cap \GL_2(\Z_p)),
\end{equation}
with $t_c \in \prod_{p<\infty} T(\Q_p)$. We can write
$$
t_c = \gamma_{c}m_{c}\kappa_{c}
$$
with $\gamma_{c} \in \GL_2(\Q)$, $m_{c} \in \GL^+_2(\R)$, and
$\kappa_{c}\in \prod_{p<\infty} \GL_2(\Z_p).$ By $(\gamma_{c})_f$ we denote the finite part of $\gamma_{c}$ when considered as an element of $\GL_2(\A)$, thus we have the equality $(\gamma_{c})_f=\gamma_{c}m_{c},$ as elements of $\GL_2(\A)$.

The matrices
$$
S_{c} = \det(\gamma_{c})^{-1}\ (\T{\gamma_{c}})S\gamma_{c}
$$
have discriminant $-d$ and form a set of representatives of the
$\SL_2(\Z)$-equivalence classes of primitive semi-integral positive
definite matrices of discriminant $-d$.

Choose $\chi$ to be a character of $T(\A)/T(\Q)T(\R)\prod_{p<\infty}
(T(\Q_p) \cap \GL_2(\Z_p))$; we identify $\chi$ with an ideal class group character of
$K$.

\begin{proposition}\label{classicalbesselperiodprop}Let $F \in S_k^{(2)}(N)$ and $S$, $\chi$, $\psi$ be as above. Then the Bessel period $B(\Phi_F)$ defined by~\eqref{defbessel} satisfies

\begin{equation}
 B({\Phi_F})= r \cdot e^{-2 \pi {\rm tr}(S)} \sum_{c \in \Cl_K} \chi(t_c)^{-1}a(F,S_c)
\end{equation}
 where the non-zero constant $r$ depends only on the normalization used for the Haar measure on  $ R(\A)$.

\end{proposition}

\begin{proof}Note that
$$
  B({\Phi_F})=\int_{\A^\times T(\Q) \backslash T(\A)}(\Phi_F)_S(t)\chi^{-1}(t)\,dt.
$$
By the coset decomposition~\eqref{e:tcosetca1}, we get (up to a non-zero constant coming from the Haar measure)
$$
   B({\Phi_F})= \sum_{c \in \Cl_K} \chi(t_c)^{-1}\int_{\R^\times \bs T(\R)}(\Phi_F)_S(t_ct_\infty)\,dt_\infty.
$$
Let us compute the inner integral. Note that $T(\R) = \C^\times$. For $g \in \GL_2$, let $\widetilde{g} =  \begin{pmatrix}
g & 0\\
0 &\!\det(g)\, \T{g^{-1}}
\end{pmatrix}$. We have
\begin{align*}
 \int_{\R^\times \bs T(\R)}(\Phi_F)_S(t_ct_\infty)\,dt_\infty
  &=\int_{\R^\times \bs T(\R)}(\Phi_F)_S(\gamma_cm_ct_\infty)\,dt_\infty\\
  &=\int_{\R^\times \bs T(\R)}(\Phi_F)_S(t_\infty  \widetilde{(\gamma_c)_f})\,dt_\infty.
\end{align*}
Put
$$
 G(Z)  = F(\gamma_c^{-1} Z\,\T{\gamma_c}^{-1} \det(\gamma_c)) = F(\widetilde{\gamma_c^{-1}}\langle Z \rangle).
$$
It is not hard to check that $G$ is a Siegel cusp form on some congruence subgroup of $\Sp_4(\Z)$. We claim that $\Phi_F(h \widetilde{(\gamma_c)_f}) = \Phi_G(h)$ for $h \in G(\A)$. By strong approximation, it suffices to prove this for $h \in G(\R)^+$. This follows from the following calculation,
\begin{align*}\Phi_F(h \widetilde{(\gamma_c)_f}) &= \Phi_F(\widetilde{m_c}h)\\&=  \mu_2(h)^k \det J((h,
  i_2))^{-k}F(\widetilde{m_c}h\langle i\rangle) \\ &= \mu_2(h)^k\det J((h,
  i_2))^{-k}F (\widetilde{\gamma_c^{-1}}h\langle i\rangle) \\ &=\mu_2(h)^k \det J((h,
  i_2))^{-k}G (h\langle i\rangle) \\&= \Phi_G(h).
\end{align*}
Thus we conclude
$$ \int_{\R^\times \bs T(\R)}(\Phi_F)_S(t_ct_\infty ) dt_\infty=\int_{\R^\times \bs T(\R)}(\Phi_G)_S(t_\infty ) dt_\infty.$$ The desired result now follows from Lemma~\ref{globalfourierlemma} and the simple observation that $a(G, S) = a(F,S_c)$.

\end{proof}

\subsection{Simultaneous non-vanishing of $L$-values}
We will now prove Theorem~\ref{th:simulnonvan}. Let $f$, $g$ be as in the statement of the theorem. In the case that $f$ and $g$ are multiples of each other, the theorem is known; indeed a stronger version easily follows from recent work of Munshi~\cite[Corollary 1]{Munshi}. So we may assume that $f$ and $g$ are not multiples of each other.

 Let $M_1$ be any divisor of $M$ with an odd number of prime factors. Let $D$ be the definite quaternion algebra over $\Q$ ramified exactly at (infinity and) the primes dividing $M_1$. Let $\pi'_f$ (resp.\ $\pi'_g$) be the Jacquet-Langlands transfer of $\pi_f$ (resp.\ $\pi_g$) to $D^\times_\A$. Using Proposition~\ref{classicalyoshidaprop}, we construct a non-zero Siegel cusp form $F_{f,g} \in S_{k+1}^{(2)}(N)$ whose adelization generates an irreducible cuspidal representation $\Pi_{f,g}$ of $G(\A)$ such that $\Pi_{f,g}$ is the global theta lift from $(D^\times_\A \times  D^\times_\A)/\A^\times\cong{\rm GSO}(D_\A) $ to $\GSp_4(\A)$ of the representation $\pi'_f \boxtimes \pi'_g$.

By assertion iii) of Proposition~\ref{classicalyoshidaprop},  $F_{f,g}$ is an eigenfunction of the operator $U(p)$ for all $p|N$. So all the required conditions for Theorem~\ref{th:nonvanfouriersiegel} are satisfied. Let $d$ be an odd squarefree integer such that there exists $T \in \Lambda_2$ with $d=-\disc(T)$ and $a(F_{f,g},T)\ne 0$. Put $K=\Q(\sqrt{-d})$. In light of Theorem~\ref{th:nonvanfouriersiegel}, it is clear that Theorem~\ref{th:simulnonvan} will be proved if we can show that for any such $d$ there exists a character $\chi \in  \widehat{\Cl_K}$ such that $L(\frac12, \pi_f \times \theta_\chi) \ne 0$ and $L(\frac12, \pi_g \times \theta_\chi) \ne 0$.

The key result which enables us to do this is the following theorem of Prasad and Takloo-Bighash as applied to our setup. We refer the reader to~\cite[Theorem 3]{prasadbighash} for the full statement. Note that by the adjointness property (proved in a much more general setting in Proposition 3.1 of \cite{PR}), for any automorphic representation $\pi$ of $D^\times(\A)$ we have the equality $$L(\frac12, \pi \times \theta_\chi) = L(\frac12, \bc_{K}(\pi) \times \chi),$$ where $\bc$ denotes base-change.

\begin{thm}[Prasad -- Takloo-Bighash]Let $D$ be a quaternion algebra and $\pi_1$, $\pi_2$ be two automorphic representations of $D^\times(\A)$ with trivial central characters. Consider $\pi_1\boxtimes \pi_2$ as an automorphic representation on the group ${\rm GSO}(D_\A) = (D^\times_\A \times  D^\times_\A)/\A^\times $ and let $\Pi$ be its theta lift to $G(\A)$. Let $d$ be an integer such that $-d$ is the discriminant of the imaginary quadratic field $K=\Q(\sqrt{-d})$ and define $S$ by~\eqref{defsd}. Let the additive character $\psi$ and the groups $T$, $R$ be defined as in Section~\ref{s:besselperiod} and let $\chi$ be a character on $T(\A)/T(\Q)$ such that $\chi |_{\A^\times}= 1$.
Then, if the linear functional on $\Pi$ given by the period integral $$\Phi \mapsto B_{S, \chi, \psi}(\Phi)$$ as defined in~\eqref{defbessel} is not identically zero, then $L(\frac12, \pi_1 \times \theta_{\chi^{-1}}) \ne 0$ and $L(\frac12, \pi_2 \times \theta_{\chi^{-1}}) \ne 0$.
\end{thm}

\begin{remark}The proof of the above theorem involves pulling back the Bessel period via theta-correspondence to ${\rm GSO}(D_\A)$. This equals a product of two toric periods on $\pi_1$ and $\pi_2$, which by Waldspurger's formula equals the product of central $L$-values. Takloo-Bighash has communicated to one of the authors that this procedure is originally due to Furusawa. \end{remark}

Now, since there exists $T \in \Lambda_2$ with $d=-\disc(T)$ and $a(F_{f,g},T)\ne 0$, it follows that we can pick a character $\chi \in \widehat{\Cl_K}$ such that $$\sum_{c \in \Cl_K} \chi(t_c)^{-1}a(F_{f,g},S_c) \ne 0.$$ For this choice of $\chi$, the Bessel period $B(\Phi_{F_{f,g}})$ is non-zero by Proposition~\ref{classicalbesselperiodprop}. Since $\Phi_{F_{f,g}}$ is a vector in $\Pi_{f,g}$ it follows that the the linear functional on $\Pi_{f,g}$ given by the period integral $$\Phi \mapsto B_{S, \chi, \psi}(\Phi)$$ is not identically zero. So, by the theorem of Prasad and Takloo-Bighash stated above,  $L(\frac12, \pi_f \times \theta_{\chi^{-1}}) = L(\frac12, \pi_f' \times \theta_{\chi^{-1}}) \ne 0$ and $L(\frac12, \pi_g \times \theta_{\chi^{-1}}) = L(\frac12, \pi_g' \times \theta_{\chi^{-1}}) \ne 0$. The proof of Theorem~\ref{th:simulnonvan} is complete.

\bibliography{yoshida-values}

\end{document}